\documentclass[11pt]{amsart}

\usepackage[T1]{fontenc}
\usepackage{lmodern}
\usepackage{microtype}
\usepackage{geometry}
\usepackage{amsmath,amssymb,amsthm,mathtools}
\usepackage{enumitem}
\usepackage[hidelinks]{hyperref}

\allowdisplaybreaks
\setlist{nosep}

\newtheorem{theorem}{Theorem}[section]
\newtheorem{proposition}[theorem]{Proposition}
\newtheorem{lemma}[theorem]{Lemma}
\newtheorem{corollary}[theorem]{Corollary}

\theoremstyle{definition}

\theoremstyle{remark}
\newtheorem{remark}[theorem]{Remark}

\newcommand{\Fq}{\mathbb F_q}

\newcommand{\C}{\mathbb C}
\newcommand{\Kl}{\operatorname{Kl}}
\newcommand{\Op}{\operatorname{Op}}

\newcommand{\charac}{\operatorname{char}}
\newcommand{\disc}{\operatorname{disc}}

\newcommand{\tr}{\operatorname{tr}}

\newcommand{\Geom}{\mathrm{geom}}

\newcommand{\1}{\mathbf 1}
\newcommand{\cK}{\mathcal K\ell_2}
\newcommand{\cF}{\mathcal F}
\newcommand{\cP}{\mathcal P}
\newcommand{\cL}{\mathcal L}

\newcommand{\olFq}{\overline{\mathbb F}_q}

\title[Spectral Incidence Bounds over Finite Fields]{Spectral Incidence Bounds for Bilinearly Coupled Polynomial Families over Finite Fields}
\author{Yao Zhi}
\date{}
\email{ericyao2026@gmail.com}

\begin{document}
\begin{abstract}
Let $\Fq$ be a finite field, let $F,G\in\Fq[T]$ be nonconstant, and let $J\in\Fq[T]$. We study the parameterized two-parameter family
\[
C_{a,b}=\{(x,t)\in\Fq^2:t=J(x)+aF(x)+bG(x)+\lambda ab\},
\qquad (a,b)\in\Fq^2,
\]
where $\lambda\in\Fq^\times$, with incidences counted with parameter multiplicity. If $d=\deg F$, $e=\deg G$, and $\charac(\Fq)>d+e$, we prove the uniform spectral estimate
\[
\left|I(\cP,\Omega)-\frac{|\cP||\Omega|}{q}\right|
\ll_{d,e}q^{5/8}\sqrt{|\cP||\Omega|}
\]
for arbitrary $\cP,\Omega\subseteq\Fq^2$. The proof combines two additive Fourier transforms with a fourth-moment estimate for Kloosterman matrices. If $1\le d,e\le2$ and $\charac(\Fq)>3$, midpoint--difference coordinates and finite Weyl quantization sharpen the error to the Vinh scale $O(q^{1/2}\sqrt{|\cP||\Omega|})$.

We further show that the latter phenomenon is intrinsic to the coefficient curve rather than to the degree of a chosen parametrization. There is a polynomial L\"uroth factorization
\[
F=f\circ H,\qquad G=g\circ H,\qquad \Fq(f,g)=\Fq(T),
\]
unique up to an affine change of $H$. If $\Gamma$ is the plane curve parametrized properly by $(f,g)$ and
\[
\kappa_H=\max_{s\in\Fq}|H^{-1}(s)|,
\]
then finite-fibre pullback multiplies the centered incidence norm by at most $\sqrt{\kappa_H}$. Consequently, if $\deg\Gamma\le2$ and $\charac(\Fq)>3$, then
\[
\left|I(\cP,\Omega)-\frac{|\cP||\Omega|}{q}\right|
\ll \sqrt{\kappa_H}\,q^{1/2}\sqrt{|\cP||\Omega|},
\]
even when $F$ and $G$ have arbitrarily large raw degree. The resulting expansion estimates for $J(x)+yF(x)+zG(x)+\lambda yz$ hold without lower bounds on the individual input sets. In the balanced regime, the $q^3$ missing-value scale for this family is already covered by a theorem of Arala--Chow; the incidence theorems here additionally yield expansion consequences in strongly unbalanced regimes.
\end{abstract}
\maketitle

\section{Introduction}

\subsection{Background and related work}

Finite-field incidence estimates provide a flexible interface between spectral graph theory, exponential sums, algebraic geometry, and additive combinatorics. The subject is closely connected with the finite-field sum--product phenomenon initiated by Bourgain, Katz, and Tao \cite{BourgainKatzTao2004}, and with Fourier-analytic incidence questions arising in finite-field distance problems, where Gauss and Kloosterman sums already play a central role \cite{IosevichRudnev2007}. A basic result of Vinh is the finite-field point--line mixing estimate
\[
\left|I(\cP,\cL)-\frac{|\cP||\cL|}{q}\right|
\le q^{1/2}\sqrt{|\cP||\cL|},
\]
valid for points and lines in $\Fq^2$ \cite{Vinh2011}. This spectral scale is a natural benchmark for finite-field incidence problems.

Several higher-degree analogues are now available. Arala and Chow proved, when $q>r$, that for graphs of univariate polynomials of degree at most $r$ one has a discrepancy of order $q^{r/2}\sqrt{|\cP||\cL|}$ \cite{AralaChow2025}. Tamo obtained the refined estimate
\[
\left|I(\cP,\cL)-\frac{|\cP||\cL|}{q}\right|
\le \sqrt{|\cP||\cL|\bigl(q+(r-1)|\cL|\bigr)}
\]
for point--polynomial incidences \cite{Tamo2025}. Kong and Tamo developed a higher-dimensional point--variety framework \cite{KongTamo2025}. More recently, Shangguan, Yang, and Zhang introduced a Cayley-color-graph approach and proved, among other results, a $q^{(D-1)/2}$ discrepancy for polynomial families contained in suitable $D$-dimensional separating subspaces \cite{ShangguanYangZhang2026}. They also note that the $q^{r/2}$ point--polynomial spectral estimate appears independently, in implicit form, in work of Mattheus, Mubayi, Nie, and Verstra\"ete on off-diagonal hypergraph Ramsey numbers \cite{MattheusMubayiNieVerstraete2025}. Arala and Chow also established a $q^{5/6}$ error for a different structured family defined by symmetric non-diagonal polynomials \cite{AralaChow2026}.

The estimates above are not linearly ordered. For instance, Tamo's bound can be stronger for small polynomial families, whereas for a family of size comparable with $q^2$ its error coefficient is of order $q$. Our first theorem gives $q^{5/8}$ for the rigid two-parameter family considered here. Thus the gain should be understood as a structured-family improvement, particularly in the large-family regime, rather than as a uniform improvement of the general point--polynomial theory. Likewise, when the three-dimensional ambient polynomial space is covered by the separating-subspace theory of \cite{ShangguanYangZhang2026}, the corresponding ambient spectral scale is $q$, while the nonlinear coefficient relation $c=\lambda ab$ exploited below yields a smaller exponent.

A second theme of the paper is that the degree of a chosen polynomial parametrization is not always the correct geometric complexity. The polynomial form of L\"uroth's theorem implies that the field $\Fq(F,G)\subseteq\Fq(T)$ is generated by a polynomial $H$, and hence $F$ and $G$ factor through a primitive polynomial parametrization of their coefficient curve; see Gutierrez, Rubio, and Sevilla \cite{GutierrezRubioSevilla2001}, and compare the classical field-decomposition viewpoint of Fried and MacRae \cite{FriedMacRae1969}. We show that the incidence operator is stable under this finite-fibre pullback. In particular, the Vinh-scale result for quadratic coefficient functions extends to arbitrary raw degrees whenever the geometric coefficient curve has degree at most two. This structural extension is independent of the Fourier and Weil-representation arguments used to prove the primitive quadratic estimate.

Polynomial expansion over finite fields is closely connected with incidence geometry. Fourier-analytic and spectral approaches appear in the work of Hart, Li, and Shen \cite{HartLiShen2013} and Vinh \cite{Vinh2013}, among others. Tao developed a structural theory of expanding polynomials over finite fields of large characteristic \cite{Tao2015}; see also the survey of de Zeeuw for the broader Elekes--R\'onyai perspective \cite{deZeeuw2018}. Pham, Vinh, and de Zeeuw obtained expansion results for three-variable polynomials, including a classification in the quadratic case \cite{PhamVinhDeZeeuw2019}. Arala and Chow subsequently proved almost-strong expansion for generic higher-degree polynomials and, for ternary polynomials of the form $au^d+\Phi(v,w)u+\Psi(v,w)$ with $\Phi,\Psi$ algebraically independent, a $q^3/(|U||V||W|)$ missing-value bound once each input set has size at least a constant multiple of $q^{2/3}$ \cite[Theorem~1.7]{AralaChow2025}. Their later work gives a $q^{11/3}/(|X||Y||Z|)$ bound for another structured higher-degree family \cite{AralaChow2026}. The expansion polynomial considered below can be rewritten in the former Arala--Chow shape after permuting the variables; accordingly, our expansion statements are presented as applications of the incidence theorems, with emphasis on the absence of individual lower-size hypotheses. The discrete oscillatory-operator viewpoint used in the low-degree part of this paper is also related in spirit to Dong's finite-field H\"ormander theorem \cite{Dong2019}.

\subsection{The family and the main results}

Fix nonconstant polynomials $F,G\in\Fq[T]$, a polynomial $J\in\Fq[T]$, and $\lambda\in\Fq^\times$. For $(a,b)\in\Fq^2$ define
\begin{equation}
C_{a,b}=\{(x,t)\in\Fq^2:t=J(x)+aF(x)+bG(x)+\lambda ab\}.
\label{eq:curve-family}
\end{equation}
For $\cP\subseteq\Fq^2$ and a parameter set $\Omega\subseteq\Fq^2$, define the parameterized incidence number
\[
I(\cP,\Omega)
=
\#\{((x,t),(a,b))\in\cP\times\Omega:(x,t)\in C_{a,b}\}.
\]
The adjective \emph{parameterized} is important: if distinct parameters define the same geometric curve, they are counted with their multiplicity.

Our first result is uniform in the degree of the common term $J$.

\begin{theorem}[Fixed-degree incidence bound]\label{thm:general-incidence}
Let $F,G\in\Fq[T]$ be nonconstant of degrees $d$ and $e$, respectively, and suppose that $p=\charac(\Fq)>d+e$. Let $J\in\Fq[T]$ and $\lambda\in\Fq^\times$. Then, for every $\cP\subseteq\Fq^2$ and every $\Omega\subseteq\Fq^2$,
\[
\left|I(\cP,\Omega)-\frac{|\cP||\Omega|}{q}\right|
\ll_{d,e}q^{5/8}\sqrt{|\cP||\Omega|}.
\]
\end{theorem}

For quadratic coefficient functions the spectral estimate improves to the point--line scale.

\begin{theorem}[Quadratic-range incidence bound]\label{thm:low-degree-incidence}
Let $1\leq\deg F,\deg G\leq2$ and suppose that $\charac(\Fq)>3$. Under the remaining hypotheses of Theorem~\ref{thm:general-incidence},
\[
\left|I(\cP,\Omega)-\frac{|\cP||\Omega|}{q}\right|
\ll q^{1/2}\sqrt{|\cP||\Omega|},
\]
with an absolute implied constant.
\end{theorem}

Theorem~\ref{thm:low-degree-incidence} has an intrinsic extension. Let
\[
\Gamma_{F,G}
=
\overline{\{(F(x),G(x)):x\in\olFq\}}
\subseteq\mathbb A^2_{\olFq}
\]
be the geometric coefficient curve. Proposition~\ref{prop:primitive-factorization} below gives a factorization
\[
F=f\circ H,
\qquad
G=g\circ H,
\qquad
\Fq(f,g)=\Fq(T),
\]
where $H\in\Fq[T]$ is unique up to an affine change of variable. Define
\[
\kappa(F,G)
=
\max_{s\in\Fq}|\{x\in\Fq:H(x)=s\}|.
\]
This quantity is well defined because affine changes of $H$ only permute its fibres over $\Fq$.

\begin{theorem}[Intrinsic quadratic coefficient-curve bound]\label{thm:intrinsic-quadratic}
Suppose that $p>3$ and
\[
\deg\Gamma_{F,G}\le2.
\]
Then, for arbitrary $J\in\Fq[T]$, $\lambda\in\Fq^\times$, $\cP\subseteq\Fq^2$, and $\Omega\subseteq\Fq^2$,
\[
\left|I(\cP,\Omega)-\frac{|\cP||\Omega|}{q}\right|
\ll
\sqrt{\kappa(F,G)}\,q^{1/2}\sqrt{|\cP||\Omega|},
\]
with an absolute implied constant.
\end{theorem}

Thus, up to the explicit finite-fibre loss $\sqrt{\kappa(F,G)}$, the Vinh-scale exponent is controlled by the intrinsic coefficient curve rather than by the raw degrees of $F$ and $G$. For example, for every nonconstant $H\in\Fq[T]$ the pair
\[
F=H,
\qquad
G=H^2
\]
has coefficient curve $V=U^2$, and Theorem~\ref{thm:intrinsic-quadratic} gives a $q^{1/2}$ discrepancy up to the factor $\sqrt{\kappa(F,G)}\le\sqrt{\deg H}$. If $H$ is a permutation polynomial, then $\kappa(F,G)=1$ and the centered incidence operator has exactly the same singular values as the primitive quadratic model; see Remark~\ref{rem:permutation-cover}.

The same transfer principle also strengthens the fixed-degree theorem when $F$ and $G$ are high-degree parametrizations of a low-degree coefficient curve.

\begin{corollary}[Intrinsic fixed-degree incidence bound]\label{cor:intrinsic-general-incidence}
Let $D=\deg\Gamma_{F,G}$ and suppose that $p>2D$. Then
\[
\left|I(\cP,\Omega)-\frac{|\cP||\Omega|}{q}\right|
\ll_D
\sqrt{\kappa(F,G)}\,q^{5/8}\sqrt{|\cP||\Omega|}
\]
for arbitrary $J$, $\lambda$, $\cP$, and $\Omega$ as above.
\end{corollary}

The associated expansion problem concerns
\begin{equation}
P(x,y,z)=J(x)+yF(x)+zG(x)+\lambda yz.
\label{eq:expander-poly}
\end{equation}
The parameterized formulation avoids any coefficient-map injectivity issue in the following corollaries.

\begin{corollary}[Unrestricted fixed-degree expansion]\label{cor:general-expansion}
Under the hypotheses of Theorem~\ref{thm:general-incidence}, let $X,Y,Z\subseteq\Fq$ be nonempty. Then
\[
q-|P(X,Y,Z)|
\ll_{d,e}
\frac{q^{13/4}}{|X||Y||Z|}.
\]
In particular, if $|X||Y||Z|\ge q^{9/4+\varepsilon}$, then
\[
|P(X,Y,Z)|=q-O_{d,e}(q^{1-\varepsilon}).
\]
\end{corollary}

\begin{corollary}[Unrestricted quadratic-range expansion]\label{cor:low-degree-expansion}
Under the hypotheses of Theorem~\ref{thm:low-degree-incidence}, let $X,Y,Z\subseteq\Fq$ be nonempty. Then
\[
q-|P(X,Y,Z)|
\ll
\frac{q^3}{|X||Y||Z|}.
\]
Consequently, if $|X||Y||Z|\ge q^{2+\varepsilon}$, then
\[
|P(X,Y,Z)|=q-O(q^{1-\varepsilon}).
\]
\end{corollary}

\begin{corollary}[Intrinsic quadratic expansion]\label{cor:intrinsic-quadratic-expansion}
Under the hypotheses of Theorem~\ref{thm:intrinsic-quadratic}, let $X,Y,Z\subseteq\Fq$ be nonempty. Then
\[
q-|P(X,Y,Z)|
\ll
\kappa(F,G)\frac{q^3}{|X||Y||Z|}.
\]
In particular, if $|X||Y||Z|\ge \kappa(F,G)q^{2+\varepsilon}$, then
\[
|P(X,Y,Z)|=q-O(q^{1-\varepsilon}).
\]
\end{corollary}

\begin{remark}[Relation with the expansion theorem of Arala--Chow]\label{rem:arala-chow-comparison}
The balanced expansion problem for \eqref{eq:expander-poly} is already covered, throughout the present family, by \cite[Theorem~1.7]{AralaChow2025}. Indeed, after permuting the variables,
\[
P(x,y,z)
=
y\bigl(F(x)+\lambda z\bigr)+\bigl(J(x)+zG(x)\bigr).
\]
Set
\[
\Phi(x,z)=F(x)+\lambda z,
\qquad
\Psi(x,z)=J(x)+zG(x).
\]
We claim that $\Phi$ and $\Psi$ are algebraically independent over $\Fq$ without any additional characteristic hypothesis. Put $U=\Phi(x,z)$. Since $\lambda\ne0$,
\[
z=\frac{U-F(x)}{\lambda},
\qquad
\Fq(x,z)=\Fq(x,U),
\]
and hence
\[
\Psi(x,z)
=
J(x)-\frac{F(x)G(x)}{\lambda}
+
U\frac{G(x)}{\lambda}.
\]
Viewed as an element of $\Fq(U)(x)$, the right-hand side is a nonconstant polynomial in the transcendental element $x$, because $G$ is nonconstant. It is therefore transcendental over $\Fq(U)$. Consequently
\[
\operatorname{trdeg}_{\Fq}\Fq(\Phi,\Psi)=2,
\]
which proves the claimed algebraic independence.

It follows from \cite[Theorem~1.7]{AralaChow2025} that whenever $|X|,|Y|,|Z|\ge Cq^{2/3}$ one has
\[
q-|P(X,Y,Z)|
\ll_{\deg \Phi,\deg \Psi}
\frac{q^3}{|X||Y||Z|}.
\]
Thus the balanced $q^3$ scale in Corollaries~\ref{cor:low-degree-expansion} and \ref{cor:intrinsic-quadratic-expansion} is not new, and the Arala--Chow theorem also gives a stronger balanced consequence than Corollary~\ref{cor:general-expansion}. The point of the expansion corollaries here is different: they are consequences of incidence estimates valid for arbitrary parameter subsets and impose no lower bound on any individual factor. For example, when $|X|=|Y|=q$, Corollary~\ref{cor:low-degree-expansion} is already nontrivial for $|Z|\ge q^\varepsilon$, while Corollary~\ref{cor:general-expansion} is nontrivial for $|Z|\ge q^{1/4+\varepsilon}$. These strongly unbalanced regimes lie outside the individual-size hypothesis of \cite[Theorem~1.7]{AralaChow2025}. The intrinsic version additionally permits arbitrarily high raw degrees when the geometric coefficient curve has degree at most two, again with the explicit finite-fibre factor $\kappa(F,G)$ in the missing-value estimate.
\end{remark}

The principal contributions are Theorems~\ref{thm:general-incidence}, \ref{thm:low-degree-incidence}, and \ref{thm:intrinsic-quadratic}. The first two are direct spectral estimates, proved by Kloosterman fourth moments and finite Weyl quantization, respectively. The third shows that the Vinh-scale exponent is stable under finite polynomial covers up to the explicit factor $\sqrt{\kappa(F,G)}$, and is therefore governed by the intrinsic coefficient curve rather than by the degree of a chosen parametrization. The expansion corollaries record product-set consequences of these incidence estimates rather than a new balanced expansion theorem.

\subsection{A geometric interpretation}

The coefficient relation has a useful three-dimensional interpretation. Set
\[
U=F(x),\qquad V=G(x),\qquad R=t-J(x).
\]
Then the incidence relation becomes
\begin{equation}
R=aU+bV+\lambda ab.
\label{eq:tangent-plane}
\end{equation}
Thus the points lie on the vertical cylinder above the polynomial curve $x\mapsto(F(x),G(x))$, while the parameter family is a two-dimensional family of affine planes.

\begin{proposition}\label{prop:tangent-planes}
The planes in \eqref{eq:tangent-plane} are exactly the affine tangent planes to the quadric
\[
\mathcal Q=\{(U,V,R)\in\Fq^3:UV+\lambda R=0\}
\]
at the points $(-\lambda b,-\lambda a,-\lambda ab)$.
\end{proposition}

\begin{proof}
Let $\Phi(U,V,R)=UV+\lambda R$. At $p_{a,b}=(-\lambda b,-\lambda a,-\lambda ab)$ one has $\Phi(p_{a,b})=0$ and
\[
\nabla\Phi(p_{a,b})=(-\lambda a,-\lambda b,\lambda).
\]
Hence the tangent plane is
\[
-a(U+\lambda b)-b(V+\lambda a)+(R+\lambda ab)=0,
\]
which is equivalent to \eqref{eq:tangent-plane}.
\end{proof}

The finite-cover principle gives a second geometric feature. If $F=f\circ H$ and $G=g\circ H$, then the map $x\mapsto(F(x),G(x))$ factors through the proper parametrization $s\mapsto(f(s),g(s))$ of the same coefficient curve. The raw degrees record both the degree of that curve and the multiplicity of the chosen parametrization. Proposition~\ref{prop:primitive-factorization} separates these two contributions.

\subsection{Outline of the proofs}

The proof of Theorem~\ref{thm:general-incidence} uses two Fourier transforms. Fourier analysis in the output coordinate reduces the centered incidence matrix to a matrix $S$ indexed by nonzero additive frequencies and $x\in\Fq$. The Gram matrix $SS^*$ has an elementary same-frequency part and a cross-frequency part depending on
\[
H_x(z)=[F(x)-F(z)][G(x)-G(z)].
\]
After inversion of the nonzero frequency and a second additive Fourier transform, the latter decomposes into Kloosterman matrices
\[
M_\tau(x,z)=\Kl\!\left(1,\frac{\tau}{\lambda}H_x(z)\right).
\]
A fourth-moment estimate gives $\|M_\tau\|\ll q^{5/4}$. The key point is that the exceptional pairs for which the two relevant Kloosterman pullbacks have the same finite ramification support form only $O_{d,e}(q)$ pairs.

For Theorem~\ref{thm:low-degree-incidence}, midpoint--difference coordinates factor $H_x(z)$ as $u^2Q(m)$ with $\deg Q\le2$. The Weyl symbol of each nonzero Kloosterman block is then a quadratic-character symbol. Constant and linear $Q$ are handled directly, the degenerate quadratic case reduces to Vinh's point--line theorem, and the nondegenerate quadratic case is treated by the finite Weil representation.

Theorem~\ref{thm:intrinsic-quadratic} is deduced from Theorem~\ref{thm:low-degree-incidence} by a separate structural argument. Section~\ref{sec:finite-cover} constructs the primitive polynomial factorization of $(F,G)$ and proves an exact factorization of the centered incidence matrix through the finite-fibre map $(x,t)\mapsto(H(x),t-J(x))$. This shows that a polynomial cover costs only the square root of its largest rational fibre.

\section{Primitive coefficient curves and finite-fibre transfer}\label{sec:finite-cover}

We isolate the structural argument that allows the spectral bounds to be stated in terms of the intrinsic coefficient curve.

\begin{proposition}[Primitive polynomial factorization]\label{prop:primitive-factorization}
Let $F,G\in\Fq[T]$ be nonconstant. Then there exist nonconstant polynomials $H,f,g\in\Fq[T]$ such that
\begin{equation}
F=f\circ H,
\qquad
G=g\circ H,
\qquad
\Fq(f,g)=\Fq(T).
\label{eq:primitive-factorization}
\end{equation}
The polynomial $H$ is unique up to an affine transformation $H\mapsto\alpha H+\beta$ with $\alpha\in\Fq^\times$ and $\beta\in\Fq$.

Let $\Gamma_{F,G}$ be the geometric coefficient curve defined above and put
\[
D=\deg\Gamma_{F,G}.
\]
For every factorization satisfying \eqref{eq:primitive-factorization},
\begin{equation}
D=\max\{\deg f,\deg g\},
\qquad
\max\{\deg F,\deg G\}=\deg H\,D.
\label{eq:degree-factorization}
\end{equation}
Moreover,
\[
\deg H=[\Fq(T):\Fq(F,G)].
\]
\end{proposition}

\begin{proof}
Set $K=\Fq(F,G)\subseteq\Fq(T)$. By the polynomial form of L\"uroth's theorem, an intermediate field of transcendence degree one that contains a nonconstant polynomial is generated by a polynomial; see \cite[Section~2]{GutierrezRubioSevilla2001}. Hence there is a nonconstant $H\in\Fq[T]$ with $K=\Fq(H)$.

Since $F,G\in\Fq(H)$, write $F=A(H)/B(H)$ in lowest terms with $A,B\in\Fq[T]$. If $B$ were nonconstant, then over $\olFq$ it would have a root $\alpha$. As $H$ is nonconstant, the equation $H(x)=\alpha$ has a solution in $\olFq$; at such a point $A(H)$ cannot vanish because $A$ and $B$ are coprime, so $F$ would have a finite pole. This contradicts $F\in\Fq[T]$. Thus $B$ is constant and $F=f(H)$ for some $f\in\Fq[T]$. The same argument gives $G=g(H)$. Since $\Fq(H)=\Fq(F,G)$, replacing the abstract variable $H$ by $T$ gives $\Fq(f,g)=\Fq(T)$, proving \eqref{eq:primitive-factorization}.

Suppose $H'$ is another polynomial generator of the same field. Then $H'$ is obtained from $H$ by an automorphism of the rational function field, so
\[
H'=\frac{aH+b}{cH+d}
\]
with $a,b,c,d\in\Fq$ and $ad-bc\ne0$. If $c\ne0$, the denominator has a root after extending scalars to $\olFq$, producing a finite pole of $H'$. Hence $c=0$, and $H'=\alpha H+\beta$ as claimed.

It remains to identify the degree of the coefficient curve. Extend scalars to $\olFq$ and let
\[
D_0=\max\{\deg f,\deg g\}.
\]
Since $\olFq(f,g)=\olFq(T)$, the map $T\mapsto(f(T),g(T))$ is birational onto $\Gamma_{F,G}$. Homogenize it to a morphism
\[
\overline\nu:\mathbb P^1\longrightarrow\mathbb P^2,
\qquad
[X:Z]\longmapsto
[Z^{D_0}:f_{D_0}(X,Z):g_{D_0}(X,Z)],
\]
where
\[
f_{D_0}(X,Z)=Z^{D_0}f(X/Z),
\qquad
g_{D_0}(X,Z)=Z^{D_0}g(X/Z).
\]
The three homogeneous coordinates have no common zero: if $Z\ne0$ this is clear from the first coordinate, while if $Z=0$ at least one of the leading coefficients of $f$ and $g$ is nonzero. Hence
\[
\overline\nu^*\mathcal O_{\mathbb P^2}(1)
\simeq
\mathcal O_{\mathbb P^1}(D_0).
\]
The morphism $\overline\nu$ is birational onto the projective closure of $\Gamma_{F,G}$, because the induced extension of function fields is $\olFq(T)/\olFq(f,g)$ of degree one. Therefore the standard degree formula for a generically finite morphism to a projective curve gives
\[
D=\deg\Gamma_{F,G}=D_0.
\]
Finally, polynomial composition gives $\deg F=(\deg f)(\deg H)$ and $\deg G=(\deg g)(\deg H)$, which yields \eqref{eq:degree-factorization}. The equality $\deg H=[\Fq(T):\Fq(H)]$ is standard for a nonconstant polynomial rational map, and $\Fq(H)=\Fq(F,G)$.
\end{proof}

\begin{remark}\label{rem:kappa-invariant}
Because the generator $H$ in Proposition~\ref{prop:primitive-factorization} is unique up to an affine change of its values, the quantity
\[
\kappa(F,G)=\max_{s\in\Fq}|H^{-1}(s)|
\]
depends only on the pair $(F,G)$. One always has
\[
1\le\kappa(F,G)\le\min\{q,\deg H\}.
\]
\end{remark}

\begin{proposition}[Finite-fibre spectral transfer]\label{prop:finite-fibre-transfer}
Let $F=f\circ H$ and $G=g\circ H$ be as in Proposition~\ref{prop:primitive-factorization}. Let $\mathsf B_{F,G,J}$ denote the centered incidence matrix of the family \eqref{eq:curve-family}, with rows indexed by $(x,t)\in\Fq^2$ and columns indexed by $(a,b)\in\Fq^2$. Let $\mathsf B_{f,g,0}$ be the analogous centered matrix for
\[
r=af(s)+bg(s)+\lambda ab.
\]
Then
\begin{equation}
\|\mathsf B_{F,G,J}\|_{2\to2}
\le
\sqrt{\kappa(F,G)}\,
\|\mathsf B_{f,g,0}\|_{2\to2}.
\label{eq:finite-cover-transfer}
\end{equation}
\end{proposition}

\begin{proof}
Define a $q^2\times q^2$ matrix $\mathsf R$, with rows indexed by $(x,t)\in\Fq^2$ and columns indexed by $(s,r)\in\Fq^2$, by
\[
\mathsf R_{(x,t),(s,r)}
=
\1_{\{s=H(x),\ r=t-J(x)\}}.
\]
Let $\mathsf A_{F,G,J}$ and $\mathsf A_{f,g,0}$ be the corresponding uncentered incidence matrices. The identities $F=f\circ H$ and $G=g\circ H$ give the exact factorization
\[
\mathsf A_{F,G,J}=\mathsf R\mathsf A_{f,g,0}.
\]
Each row of $\mathsf R$ contains exactly one nonzero entry, so $\mathsf R\1=\1$. Consequently the same factorization holds after subtracting the density-$q^{-1}$ rank-one term:
\begin{equation}
\mathsf B_{F,G,J}=\mathsf R\mathsf B_{f,g,0}.
\label{eq:centered-transfer}
\end{equation}
For a fixed $(s,r)$, the preimages under $(x,t)\mapsto(H(x),t-J(x))$ are in bijection with the solutions of $H(x)=s$. It follows that $\mathsf R^*\mathsf R$ is diagonal and
\[
(\mathsf R^*\mathsf R)_{(s,r),(s,r)}=|H^{-1}(s)|.
\]
Thus $\|\mathsf R\|_{2\to2}=\sqrt{\kappa(F,G)}$. Taking norms in \eqref{eq:centered-transfer} proves \eqref{eq:finite-cover-transfer}.
\end{proof}

\begin{proof}[Proof of Corollary~\ref{cor:intrinsic-general-incidence}]
Use Proposition~\ref{prop:primitive-factorization}. By \eqref{eq:degree-factorization}, the primitive polynomials $f$ and $g$ have degrees at most $D$. The assumption $p>2D$ implies
\[
p>\deg f+\deg g.
\]
Theorem~\ref{thm:general-incidence}, applied to the primitive family with common term zero, gives
\[
\|\mathsf B_{f,g,0}\|_{2\to2}\ll_D q^{5/8}.
\]
Proposition~\ref{prop:finite-fibre-transfer} yields
\[
\|\mathsf B_{F,G,J}\|_{2\to2}
\ll_D
\sqrt{\kappa(F,G)}\,q^{5/8}.
\]
Pairing with the indicator vectors of $\cP$ and $\Omega$ proves the corollary.
\end{proof}

\begin{proof}[Proof of Theorem~\ref{thm:intrinsic-quadratic}]
Again use Proposition~\ref{prop:primitive-factorization}. Since $\deg\Gamma_{F,G}\le2$, equation \eqref{eq:degree-factorization} gives
\[
1\le\deg f,\deg g\le2.
\]
Theorem~\ref{thm:low-degree-incidence} applied to the primitive family gives
\[
\|\mathsf B_{f,g,0}\|_{2\to2}\ll q^{1/2}.
\]
The conclusion follows from Proposition~\ref{prop:finite-fibre-transfer} and Cauchy--Schwarz.
\end{proof}

\begin{remark}[Permutation covers]\label{rem:permutation-cover}
If the polynomial $H$ in Proposition~\ref{prop:primitive-factorization} induces a permutation of $\Fq$, then $\mathsf R$ in the proof of Proposition~\ref{prop:finite-fibre-transfer} is a permutation matrix. Hence $\mathsf B_{F,G,J}$ and $\mathsf B_{f,g,0}$ have the same singular values. Thus high raw degree need not produce any spectral loss. For instance, if $F=H$ and $G=H^2$ with $H$ a permutation polynomial, the centered incidence spectrum is exactly that of the primitive pair $(T,T^2)$.
\end{remark}

\section{Preliminaries}\label{sec:prelim}

\subsection{Fourier analysis and Kloosterman sums}

Fix a nontrivial additive character $\psi:\Fq\to\C^\times$. We use the unitary Fourier transform
\[
\widehat f(\xi)=q^{-1/2}\sum_{x\in\Fq}f(x)\psi(-\xi x).
\]
For $a\in\Fq$, define the classical rank-two Kloosterman sum
\begin{equation}
\Kl(1,a)=\sum_{s\in\Fq^\times}\psi\!\left(s+\frac{a}{s}\right).
\label{eq:kl-definition}
\end{equation}
We use the standard Weil estimate $|\Kl(1,a)|\le2q^{1/2}$ for $a\ne0$ and the identity $\Kl(1,0)=-1$; in particular, the uniform bound $|\Kl(1,a)|\le2q^{1/2}$ holds for every $a\in\Fq$. These follow from the theory of Kloosterman sums; in the sheaf-theoretic form needed below, see Katz \cite{Katz1988}. Notice also that $\Kl(1,a)$ is real: complex conjugation followed by $s\mapsto-s$ leaves \eqref{eq:kl-definition} unchanged.

Let $\chi$ denote the quadratic character of $\Fq^\times$, extended by $\chi(0)=0$, and put
\[
\mathfrak g=\sum_{x\in\Fq}\psi(x^2).
\]
For $q$ odd, $|\mathfrak g|=q^{1/2}$ and $\mathfrak g^2=\chi(-1)q$. We shall repeatedly use the quadratic Gauss identity
\[
\sum_{x\in\Fq}\psi(ax^2+bx)
=
\chi(a)\mathfrak g\,\psi\!\left(-\frac{b^2}{4a}\right),
\qquad a\ne0.
\]

\subsection{A weighted point--line estimate}

We record the weighted form of Vinh's estimate that will be used in Section~\ref{sec:lowdegree}. It is the operator-norm form of the spectral proof of \cite{Vinh2011}.

\begin{lemma}[Weighted Vinh inequality]\label{lem:weighted-vinh}
Let $w_P:\Fq^2\to\C$ and let $w_L$ be a complex weight on the set of affine lines in $\Fq^2$. Then
\[
\left|
\sum_{\ell}\sum_{p\in\ell}w_P(p)w_L(\ell)
-
\frac{\bigl(\sum_pw_P(p)\bigr)\bigl(\sum_\ell w_L(\ell)\bigr)}{q}
\right|
\le q^{1/2}\|w_P\|_2\|w_L\|_2.
\]
\end{lemma}

\begin{proof}
The centered point--line incidence operator has second singular value at most $q^{1/2}$; this is exactly the spectral statement underlying Vinh's theorem \cite{Vinh2011}. Pairing this operator with the two weight vectors gives the asserted inequality by Cauchy--Schwarz.
\end{proof}

\subsection{Sheaf-theoretic input}

We shall use standard facts about Deligne's rank-two Kloosterman sheaf $\cK$ on $\mathbb G_m$. Up to a fixed sign, its trace function is \eqref{eq:kl-definition}; it is geometrically irreducible and pure of weight $1$, its local monodromy at $0$ is a single nontrivial tame unipotent Jordan block, its Swan conductor at $\infty$ equals $1$, and its geometric monodromy group is $\mathrm{SL}_2$. We use Katz's treatment of Kloosterman sheaves and their local and global monodromy, especially Chapters~4, 10, and 11 of \cite{Katz1988}. We use Deligne's Riemann hypothesis over finite fields in the form of square-root cancellation for complete sums of trace functions with no geometrically trivial constituent \cite{Deligne1980}; see also \cite{FouvryKowalskiMichel2015} for a convenient trace-function framework.

We isolate the consequence needed later.

\begin{lemma}[Two-pullback correlation]\label{lem:sheaf-correlation}
Let $R_1,R_2\in\Fq[T]$ be nonconstant polynomials of degrees at most $n<p=\charac(\Fq)$. Assume that the supports of their geometric zero divisors are distinct. Then, uniformly for $c\in\Fq^\times$,
\[
\sum_{y\in\Fq}\Kl(1,cR_1(y))\Kl(1,cR_2(y))
\ll_n q^{3/2}.
\]
\end{lemma}

\begin{proof}
For $i=1,2$, let $\cF_i=(cR_i)^*\cK$ on the largest open subset of $\mathbb A^1$ on which it is lisse. Because $R_i$ is nonconstant and $\deg R_i<p$, one has $R_i'\not\equiv0$; hence the finite morphism $R_i:\mathbb A^1\to\mathbb A^1$ is generically separable. Choose a nonempty open set $V_i\subset\mathbb G_m$ avoiding all branch values of $R_i$, and put $U_i=R_i^{-1}(V_i)$. Then $R_i:U_i\to V_i$ is a finite \`etale morphism. Since $U_i$ is obtained from $\mathbb A^1$ by deleting finitely many geometric points, it is geometrically connected; hence this finite \`etale cover is connected. It follows that the image of $\pi_1^{\Geom}(U_i)$ has finite index in $\pi_1^{\Geom}(V_i)$. Since the open immersion $V_i\hookrightarrow\mathbb G_m$ induces a surjection on geometric fundamental groups, its image in $\pi_1^{\Geom}(\mathbb G_m)$ is again of finite index. The geometric monodromy of $\cK$ is the connected group $\mathrm{SL}_2$, so restriction to this finite-index subgroup still has Zariski closure $\mathrm{SL}_2$. In particular, each $\cF_i$ is geometrically irreducible.

Let $S_i$ be the support of the zero divisor of $R_i$. Choose $r\in S_1\setminus S_2$, after interchanging the indices if necessary. If the multiplicity of $r$ as a zero of $R_1$ is $m$, then $1\le m<p$, so in particular $p\nmid m$. The local tame unipotent monodromy $U=I+N$ of $\cK$ at $0$ pulls back to $U^m=I+mN$, which is still nontrivial. Thus $\cF_1$ is ramified at $r$, whereas $\cF_2$ is lisse there. Dualization does not change the lisse locus, so $\cF_1^\vee$ is also ramified at $r$ and therefore cannot be geometrically isomorphic to $\cF_2$. By geometric irreducibility and Schur's lemma,
\[
(\cF_1\otimes\cF_2)^{\pi_1^{\Geom}}=0.
\]
Equivalently, the compactly supported cohomology in degree $2$ has no contribution to the complete trace sum.

The conductors of $\cF_i$ are $O_n(1)$. Indeed, there are at most $n$ finite ramification points, all finite ramification above $0$ is tame, and the polynomial map has tame ramification index at most $n$ at $\infty$ because $p>n$; hence the Swan conductor at $\infty$ is $O(n)$. Deligne's theorem therefore gives an $O_n(q^{3/2})$ bound for the sum over the common lisse locus, since $\cF_1\otimes\cF_2$ has weight $2$. The omitted zeros contribute $O_n(q)$ by the Weil bound for Kloosterman sums, and are absorbed by the same estimate.
\end{proof}

\subsection{Finite Weyl quantization}

Assume in this subsection that $q$ is odd. For a symbol $a:\Fq^2\to\C$, define its Weyl quantization by
\begin{equation}
(\Op^W(a)f)(x)
=
\frac1q\sum_{y,\xi\in\Fq}
 a\!\left(\frac{x+y}{2},\xi\right)
 \psi(\xi(x-y))f(y).
\label{eq:weyl-quantization}
\end{equation}
Thus, if a matrix $K(x,y)$ has Weyl symbol
\[
\sigma_K(m,\xi)=\sum_{u\in\Fq}K\!\left(m+\frac u2,m-\frac u2\right)\psi(-\xi u),
\]
then $K=\Op^W(\sigma_K)$.

We use one standard structural consequence of the finite Weil representation. The invariant Cayley-transform kernel for the Weil representation over finite fields of odd characteristic is given in the ansatz (2.1.3) and Theorem~2.1.5 of \cite{GurevichHadani2007}; maximal-torus restrictions and the associated character-space decompositions are discussed in Section~3.4, in particular Theorem~11, of \cite{GurevichHadani2010}.

\begin{lemma}[Quadratic Weyl phases and one Weil torus]\label{lem:weil-torus-input}
Let $A\in\Fq^\times$ and, for $t\in\Fq^\times$, put
\[
U_t=\Op^W\bigl(\psi(t(Am^2-\xi^2))\bigr).
\]
Set
\[
L_A=
\begin{pmatrix}
0&1\\
A&0
\end{pmatrix},
\qquad
T_A=Z_{\mathrm{SL}_2}(L_A),
\qquad
E_A=\{t\in\Fq^\times:At^2=1\}.
\]
Then $T_A$ is a maximal torus. For every $t\in\Fq^\times\setminus E_A$, define
\begin{equation}
 g_t=(I+tL_A)(I-tL_A)^{-1}
 =\frac{1}{1-At^2}
 \begin{pmatrix}
 1+At^2&2t\\
 2At&1+At^2
 \end{pmatrix}.
\label{eq:gt-cayley}
\end{equation}
There is a unit-modulus scalar $\gamma_A(t)$ such that
\begin{equation}
U_t=\gamma_A(t)\rho(g_t),
\label{eq:Ut-Weil}
\end{equation}
where $\rho$ is the Schr\"odinger Weil representation. Moreover, for every character $\nu$ of $T_A(\Fq)$, the function
\[
t\longmapsto \gamma_A(t)\nu(g_t)
\]
on $\{t\in\Fq^\times:At^2\ne1\}$ is the trace function of a tame rank-one sheaf of conductor bounded by an absolute constant, uniformly in $A$, $q$, and $\nu$. If $At^2=1$, then $U_t$ has rank one and $\|U_t\|=q^{1/2}$.
\end{lemma}

\begin{proof}
A direct calculation from \eqref{eq:weyl-quantization} gives
\begin{equation}
U_t(x,y)
=
\frac{\mathfrak g}{q}\chi(-t)
\psi\!\left(
\frac{At^2+1}{4t}(x^2+y^2)
+
\frac{At^2-1}{2t}xy
\right).
\label{eq:Ut-kernel}
\end{equation}
If $At^2=1$, the mixed term vanishes, and \eqref{eq:Ut-kernel} factors as a product of one function of $x$ and one function of $y$. Hence $U_t$ has rank one. Since each row has Euclidean norm $1$, its unique nonzero singular value is $q^{1/2}$.

Assume now that $At^2\ne1$. The standard Schr\"odinger kernel for the Weil operator attached to a matrix
$g=\left(\begin{smallmatrix}a&b\\ c&d\end{smallmatrix}\right)$ with $b\ne0$ is, up to a unit scalar,
\[
q^{-1/2}\psi\!\left(\frac{d x^2-2xy+a y^2}{2b}\right);
\]
see \cite{GurevichHadani2007}. For the matrix $g_t$ in \eqref{eq:gt-cayley}, one has
\[
a=d=\frac{1+At^2}{1-At^2},
\qquad
b=\frac{2t}{1-At^2}.
\]
Its kernel therefore has exactly the quadratic phase in \eqref{eq:Ut-kernel}, up to the harmless convention-dependent replacement of $g_t$ by $g_t^{-1}$ and a unit scalar. Absorbing this scalar into $\gamma_A(t)$ gives \eqref{eq:Ut-Weil}. Since $g_t$ is a rational function of $L_A$, it commutes with $L_A$ and hence belongs to $T_A$. Conversely,
\[
T_A=
\left\{
\begin{pmatrix}u&v\\ Av&u\end{pmatrix}:u^2-Av^2=1
\right\},
\]
so $T_A$ is the one-dimensional maximal torus centralizing the regular semisimple element $L_A$.

It remains to justify the uniform tame-conductor assertion. Choose $s\in\overline{\Fq}$ with $s^2=A$. Over $\overline{\Fq}$, the two eigenvalues of $g_t$ are
\begin{equation}
r_t=\frac{1+st}{1-st},
\qquad r_t^{-1}.
\label{eq:torus-cayley-coordinate}
\end{equation}
If $A$ is a square in $\Fq$, this identifies $T_A$ with the split torus $\mathbb G_m$, and every character $\nu$ of $T_A(\Fq)$ pulls back along \eqref{eq:torus-cayley-coordinate} to a Kummer character sheaf. If $A$ is nonsquare, the same description holds after base change to $\Fq(s)$; the $\Fq$-points form the norm-one torus. In this case a character of $T_A(\Fq)$ is realized, via the Lang isogeny, by a rank-one character sheaf on the algebraic torus $T_A$; this is the standard Lang-torsor construction for finite-field character sheaves, see Deligne \cite[Sommes trig., \S1.4]{Deligne1977}. Its finite monodromy has order dividing $|T_A(\Fq)|=q+1$, which is prime to $p$; consequently its local monodromy at the boundary of a smooth compactification is tame. In either case, after pullback by $t\mapsto g_t$, the only possible singularities occur at a fixed set of at most four geometric points, namely $0$, $\infty$, and the two roots of $1-At^2$.

We also make the scalar in \eqref{eq:Ut-Weil} explicit enough for the conductor estimate. In the standard Schr\"odinger normalization, the kernel of $\rho(g)$ for $b\ne0$ has scalar prefactor $q^{-1/2}\varepsilon_\psi\chi(\beta b)$, where $\varepsilon_\psi$ has modulus one and $\beta\in\Fq^\times$ depends only on the fixed Fourier convention; compare the invariant kernel formula (2.1.3) and Theorem~2.1.5 of \cite{GurevichHadani2007}. For $g_t$ one has $b_t=2t/(1-At^2)$, whereas \eqref{eq:Ut-kernel} has scalar $q^{-1/2}(\mathfrak g/q^{1/2})\chi(-t)$. Their ratio therefore has the form
\[
\gamma_A(t)=\varepsilon_\psi'\chi(\beta'(1-At^2)),
\]
for constants $|\varepsilon_\psi'|=1$ and $\beta'\in\Fq^\times$ independent of $t$ and $A$; changing the Weil-representation convention only changes these constants. Hence $\gamma_A$ is itself a tame quadratic Kummer trace function with possible ramification only at the two roots of $1-At^2$ and at infinity. Combining this with the preceding description of $\nu(g_t)$ shows that $\gamma_A(t)\nu(g_t)$ is the trace function of a tame rank-one sheaf. Its rank is one, its possible geometric singularities lie among the fixed set $\{0,\infty\}\cup\{t:1-At^2=0\}$, and all local monodromy at these boundary points is tame. Hence its conductor is bounded by an absolute constant, uniformly in $A$, $q$, and $\nu$. This proves the lemma.
\end{proof}

\section{The fixed-degree incidence theorem}\label{sec:general-proof}

\subsection{The first Fourier transform}

Let $\mathsf A$ be the full incidence matrix whose rows are indexed by $(x,t)\in\Fq^2$ and whose columns are indexed by $(a,b)\in\Fq^2$:
\[
\mathsf A_{(x,t),(a,b)}
=
\1_{\{t=J(x)+aF(x)+bG(x)+\lambda ab\}}.
\]
Let
\[
\mathsf B=\mathsf A-q^{-1}\1\1^*.
\]

\begin{lemma}[Output-frequency reduction]\label{lem:output-fourier}
Let $\mathsf S$ be the matrix with rows indexed by $(\xi,x)\in\Fq^\times\times\Fq$, columns indexed by $(a,b)\in\Fq^2$, and entries
\begin{equation}
\mathsf S_{(\xi,x),(a,b)}
=
\psi\bigl(-\xi(J(x)+aF(x)+bG(x)+\lambda ab)\bigr).
\label{eq:S-definition}
\end{equation}
Then
\begin{equation}
\|\mathsf B\|=q^{-1/2}\|\mathsf S\|.
\label{eq:centered-normalization}
\end{equation}
\end{lemma}

\begin{proof}
Apply the unitary Fourier transform in the $t$-coordinate to the rows of $\mathsf A$. For frequency $\xi$ the transformed entry equals
\[
q^{-1/2}\psi\bigl(-\xi(J(x)+aF(x)+bG(x)+\lambda ab)\bigr).
\]
At $\xi=0$ this is the constant $q^{-1/2}$, exactly the transform of the rank-one matrix $q^{-1}\1\1^*$. Thus the zero-frequency block cancels in $\mathsf B$, while the nonzero blocks are $q^{-1/2}\mathsf S$. Unitarity gives \eqref{eq:centered-normalization}.
\end{proof}

\subsection{The Gram matrix}

For $x,z\in\Fq$ define
\[
H_x(z)=[F(x)-F(z)][G(x)-G(z)].
\]

\begin{lemma}[Gram decomposition]\label{lem:gram-decomposition}
There is a diagonal unitary matrix $\mathsf U$ on $\Fq^\times\times\Fq$ such that
\[
\mathsf U^*\mathsf S\mathsf S^*\mathsf U=\mathsf D+q\mathsf H,
\]
where $\mathsf D$ is block diagonal in the frequency variable and
\[
\|\mathsf D\|\le dq^2,
\]
while, after writing $\alpha=\xi^{-1}$ and $\beta=\eta^{-1}$,
\[
\mathsf H((\alpha,x),(\beta,z))
=
\begin{cases}
\displaystyle
\psi\!\left(-\frac{H_x(z)}{\lambda(\alpha-\beta)}\right),&\alpha\ne\beta,\\[6pt]
0,&\alpha=\beta.
\end{cases}
\]
Here $\alpha,\beta\in\Fq^\times$.
\end{lemma}

\begin{proof}
For rows $(\xi,x)$ and $(\eta,z)$, expanding \eqref{eq:S-definition} gives
\[
(\mathsf S\mathsf S^*)((\xi,x),(\eta,z))
=
\psi(-\xi J(x)+\eta J(z))
\sum_{a,b}\psi(Aa+Bb+Cab),
\]
where
\[
A=\eta F(z)-\xi F(x),\qquad
B=\eta G(z)-\xi G(x),\qquad
C=\lambda(\eta-\xi).
\]
If $\xi=\eta$, the sum over $a,b$ equals $q^2$ precisely when $F(x)=F(z)$ and $G(x)=G(z)$, and vanishes otherwise. Hence each same-frequency block is, up to diagonal unitary conjugation, $q^2$ times the matrix of the equivalence relation $x\sim z$ if $(F(x),G(x))=(F(z),G(z))$. Each fibre has size at most $d$, since $F$ is nonconstant of degree $d$, so the norm of every such block is at most $dq^2$.

Assume now that $\xi\ne\eta$. The elementary identity
\[
\sum_{a,b}\psi(Cab+Aa+Bb)=q\psi(-AB/C)
\]
gives the required sum. The algebraic identity
\[
\frac{AB}{\eta-\xi}
=
\eta F(z)G(z)-\xi F(x)G(x)
+
\frac{\xi\eta}{\eta-\xi}H_x(z)
\]
shows that the first two terms, together with the $J$-terms, are removed by the diagonal unitary with entries
\[
\psi\!\left(-\xi J(x)+\frac{\xi F(x)G(x)}{\lambda}\right).
\]
The remaining phase is
\[
\psi\!\left(-\frac{\xi\eta}{\lambda(\eta-\xi)}H_x(z)\right).
\]
Since $\xi\eta/(\eta-\xi)=1/(\alpha-\beta)$, the stated formula follows.
\end{proof}

\subsection{The second Fourier transform}

Extend $\mathsf H$ from $\Fq^\times\times\Fq$ to a Hermitian matrix $\widetilde{\mathsf H}$ on $\Fq\times\Fq$ by the same formula, with zero diagonal in the first coordinate. The original $\mathsf H$ is a principal submatrix, hence
\begin{equation}
\|\mathsf H\|\le\|\widetilde{\mathsf H}\|.
\label{eq:principal-submatrix}
\end{equation}
Because the kernel depends on the first coordinates only through their difference, the additive Fourier transform in that coordinate diagonalizes $\widetilde{\mathsf H}$.

\begin{lemma}[Kloosterman block decomposition]\label{lem:kloosterman-blocks}
The matrix $\widetilde{\mathsf H}$ is unitarily equivalent to the direct sum of $q$ matrices $M_\tau$, $\tau\in\Fq$, where
\begin{equation}
M_\tau(x,z)
=
\sum_{h\in\Fq^\times}
\psi\!\left(-\tau h-\frac{H_x(z)}{\lambda h}\right).
\label{eq:M-tau-sum}
\end{equation}
For $\tau\ne0$,
\begin{equation}
M_\tau(x,z)
=
\Kl\!\left(1,\frac{\tau}{\lambda}H_x(z)\right),
\label{eq:kloosterman-block}
\end{equation}
and
\begin{equation}
M_0(x,z)=q\1_{\{H_x(z)=0\}}-1.
\label{eq:zero-block}
\end{equation}
Moreover, $\|M_0\|\ll_{d,e}q$.
\end{lemma}

\begin{proof}
The block formula \eqref{eq:M-tau-sum} is the Fourier transform of the convolution kernel in the first coordinate. If $\tau\ne0$, the substitution $s=-\tau h$ turns the phase into $s+(\tau/\lambda)H_x(z)/s$, proving \eqref{eq:kloosterman-block}. If $\tau=0$, inversion on $\Fq^\times$ gives
\[
\sum_{h\ne0}\psi\!\left(-\frac{u}{\lambda h}\right)
=
\begin{cases}
q-1,&u=0,\\
-1,&u\ne0,
\end{cases}
\]
which is \eqref{eq:zero-block}.

Let $E(x,z)=\1_{\{H_x(z)=0\}}$. Since $H_x(z)=0$ if and only if $F(z)=F(x)$ or $G(z)=G(x)$, every row and every column of $E$ contains at most $d+e$ ones. Schur's test gives $\|E\|\le d+e$, while the all-ones matrix has norm $q$. Thus \eqref{eq:zero-block} implies $\|M_0\|\le(d+e+1)q$.
\end{proof}

\subsection{Fourth moments of the nonzero blocks}

For $x\in\Fq$, let $S_x$ denote the support of the zero divisor of the polynomial $H_x(T)$ over $\olFq$; thus multiplicities are discarded.

\begin{lemma}[Exceptional support pairs]\label{lem:exceptional-pairs}
One has
\begin{equation}
\#\{(x,z)\in\Fq^2:S_x=S_z\}\le(d+e)q.
\label{eq:exceptional-count}
\end{equation}
\end{lemma}

\begin{proof}
If $S_x=S_z$, then $z\in S_z=S_x$ because $H_z(z)=0$. Hence $H_x(z)=0$, so $F(z)=F(x)$ or $G(z)=G(x)$. For each fixed $x$, the two equations have at most $d$ and $e$ solutions, respectively. Summing over $x$ gives \eqref{eq:exceptional-count}.
\end{proof}

\begin{proposition}[Fourth-moment bound]\label{prop:fourth-moment}
For every $\tau\in\Fq^\times$,
\begin{equation}
\tr(M_\tau^4)\ll_{d,e}q^5.
\label{eq:fourth-moment-bound}
\end{equation}
Consequently,
\begin{equation}
\|M_\tau\|\ll_{d,e}q^{5/4}.
\label{eq:M-norm-general}
\end{equation}
\end{proposition}

\begin{proof}
Fix $\tau\ne0$ and put $c=\tau/\lambda$. Since $H_x(z)=H_z(x)$ and Kloosterman sums are real, $M_\tau$ is Hermitian. Hence
\begin{equation}
\tr(M_\tau^4)
=
\sum_{x,z\in\Fq}|(M_\tau^2)(x,z)|^2.
\label{eq:fourth-moment-hs}
\end{equation}
Using \eqref{eq:kloosterman-block} and the symmetry $H_y(z)=H_z(y)$,
\[
(M_\tau^2)(x,z)
=
\sum_{y\in\Fq}
\Kl(1,cH_x(y))\Kl(1,cH_z(y)).
\]
If $S_x\ne S_z$, Lemma~\ref{lem:sheaf-correlation} gives
\[
|(M_\tau^2)(x,z)|\ll_{d,e}q^{3/2}.
\]
For the pairs with $S_x=S_z$, the pointwise Weil bound gives
\[
|(M_\tau^2)(x,z)|\le4q^2.
\]
By Lemma~\ref{lem:exceptional-pairs}, the exceptional set has $O_{d,e}(q)$ elements. Substitution in \eqref{eq:fourth-moment-hs} yields
\[
\tr(M_\tau^4)
\ll_{d,e}q^2q^3+q q^4
\ll_{d,e}q^5,
\]
which is \eqref{eq:fourth-moment-bound}. Since $M_\tau$ is Hermitian,
\[
\|M_\tau\|^4\le\tr(M_\tau^4),
\]
and \eqref{eq:M-norm-general} follows.
\end{proof}

\begin{proof}[Proof of Theorem~\ref{thm:general-incidence}]
By Lemma~\ref{lem:kloosterman-blocks}, Proposition~\ref{prop:fourth-moment}, and the direct-sum decomposition,
\[
\|\widetilde{\mathsf H}\|\ll_{d,e}q^{5/4}.
\]
Together with \eqref{eq:principal-submatrix} and Lemma~\ref{lem:gram-decomposition}, this gives
\[
\|\mathsf S\mathsf S^*\|
\ll_{d,e}q^2+q\,q^{5/4}
\ll_{d,e}q^{9/4}.
\]
Hence $\|\mathsf S\|\ll_{d,e}q^{9/8}$. Lemma~\ref{lem:output-fourier} gives
\[
\|\mathsf B\|\ll_{d,e}q^{5/8}.
\]
If $1_{\cP}$ and $1_\Omega$ are the indicator vectors of arbitrary subsets $\cP\subseteq\Fq^2$ and $\Omega\subseteq\Fq^2$, then
\[
I(\cP,\Omega)-\frac{|\cP||\Omega|}{q}
=
\langle1_{\cP},\mathsf B1_\Omega\rangle.
\]
Cauchy--Schwarz finishes the proof.
\end{proof}

\section{Unrestricted expansion consequences in fixed degree}\label{sec:general-expansion}

\begin{proof}[Proof of Corollary~\ref{cor:general-expansion}]
Let
\[
W=\Fq\setminus P(X,Y,Z).
\]
If $W=\varnothing$, there is nothing to prove. Assume henceforth that $W\ne\varnothing$. Use the point set $\cP=X\times W$ and the parameter set $\Omega=Y\times Z$. By the definition of $W$, no point $(x,w)\in\cP$ is incident to a parameter $(y,z)\in\Omega$ for the family \eqref{eq:curve-family}; hence $I(\cP,\Omega)=0$. Theorem~\ref{thm:general-incidence} gives
\[
\frac{|X||W||Y||Z|}{q}
\ll_{d,e}
q^{5/8}\sqrt{|X||W||Y||Z|}.
\]
After squaring and cancelling the common nonzero factor $|X||W||Y||Z|$, and then dividing by $|X||Y||Z|$, we obtain
\[
|W|\ll_{d,e}\frac{q^{13/4}}{|X||Y||Z|}.
\]
The product-size consequence follows immediately from the assumed lower bound on $|X||Y||Z|$.
\end{proof}

\section{The quadratic-range improvement}\label{sec:lowdegree}

Assume throughout this section that $1\le\deg F,\deg G\le2$ and $p>3$. We retain the notation from Section~\ref{sec:general-proof}.

\subsection{Midpoint--difference factorization and the Weyl symbol}

Write
\[
x=m+\frac u2,\qquad z=m-\frac u2.
\]
Since $F$ and $G$ have degree at most $2$,
\[
F(x)-F(z)=uF'(m),\qquad G(x)-G(z)=uG'(m).
\]
Thus
\begin{equation}
H_x(z)=u^2Q(m),
\qquad
Q(m)=F'(m)G'(m),
\qquad
\deg Q\le2.
\label{eq:midpoint-factorization}
\end{equation}
The polynomial $Q$ is nonzero because $F$ and $G$ are nonconstant and $p>2$.

\begin{lemma}[Weyl symbol of a Kloosterman block]\label{lem:weyl-symbol}
Let $c\in\Fq^\times$ and define
\[
M_c(x,z)=\Kl(1,cH_x(z)).
\]
Put
\[
b_c(m,\xi)=\chi(-1)\chi(4cQ(m)-\xi^2).
\]
Then
\begin{equation}
M_c
=
q\Op^W(b_c)-q\sum_{r:Q(r)=0}P_r,
\label{eq:weyl-block-identity}
\end{equation}
where $P_r$ is the orthogonal projection onto the coordinate vector $\delta_r$.
\end{lemma}

\begin{proof}
Let $\sigma_c$ be the Weyl symbol of $M_c$. By \eqref{eq:midpoint-factorization},
\[
\sigma_c(m,\xi)
=
\sum_{u\in\Fq}\Kl(1,cu^2Q(m))\psi(-\xi u).
\]
If $Q(m)\ne0$, expand the Kloosterman sum and evaluate the $u$-sum by the quadratic Gauss formula:
\[
\begin{aligned}
\sigma_c(m,\xi)
&=
\sum_{s\ne0}\psi(s)
\sum_u\psi\!\left(\frac{cQ(m)}s u^2-\xi u\right)\\
&=
\mathfrak g\chi(cQ(m))
\sum_{s\ne0}\chi(s)
\psi\!\left(s\left(1-\frac{\xi^2}{4cQ(m)}\right)\right)\\
&=
q\chi(-1)\chi(4cQ(m)-\xi^2).
\end{aligned}
\]
The last identity also covers the case $4cQ(m)=\xi^2$, since the quadratic character is extended by zero. If $Q(m)=0$, then $M_c(m+u/2,m-u/2)=\Kl(1,0)=-1$, so
\[
\sigma_c(m,\xi)=-q\1_{\{\xi=0\}}.
\]
At such an $m$, the difference between this symbol and $q b_c(m,\xi)$ is the constant $-q$ for every $\xi$. Under Weyl inversion this constant symbol at the single midpoint $m=r$ is exactly $-qP_r$. Summing over the roots of $Q$ proves \eqref{eq:weyl-block-identity}.
\end{proof}

Thus it remains to prove a uniform $O(1)$ bound for $\Op^W(b_c)$. We treat the possible shapes of $Q$ separately.

\subsection{Constant and linear symbols}

If $Q$ is constant, then $b_c$ is independent of $m$, so $\Op^W(b_c)$ is a Fourier multiplier and
\[
\|\Op^W(b_c)\|\le1.
\]
The linear case is also elementary, though less immediate.

\begin{lemma}[Linear quadratic-character symbol]\label{lem:linear-symbol}
Assume $p>3$. For every $a\in\Fq^\times$,
\[
\left\|\Op^W\bigl(\chi(am-\xi^2)\bigr)\right\|\ll1.
\]
The same conclusion holds with $am$ replaced by any nonconstant affine-linear function of $m$.
\end{lemma}

\begin{proof}
A translation in the physical variable conjugates Weyl quantization by translation and removes the constant term of an affine-linear function, so it suffices to consider $am$. Put
\[
B_a=\Op^W\bigl(\chi(am-\xi^2)\bigr),
\]
and let $e_\eta(x)=q^{-1/2}\psi(\eta x)$ be the Fourier basis. Expanding the quadratic character by a Gauss sum and summing first in the physical variable gives
\[
B_ae_\eta
=
\mathfrak g^{-1}
\sum_{t\ne0}
\chi(t)\psi\!\left(-t\left(\eta+\frac{at}{2}\right)^2\right)e_{\eta+at}.
\]
Hence, for $\nu\ne\eta$,
\[
\langle B_ae_\eta,e_\nu\rangle
=
\mathfrak g^{-1}
\chi\!\left(\frac{\nu-\eta}{a}\right)
\psi\!\left(
-\frac{\nu-\eta}{a}\left(\frac{\nu+\eta}{2}\right)^2
\right),
\]
and the diagonal entries vanish.

Write $\nu=s+d$, $\mu=s-d$. A direct multiplication of the Fourier-basis matrix yields
\[
(B_aB_a^*)(\nu,\mu)
=
\frac1q
\psi\!\left(-\frac{d(d^2+4s^2)}{2a}\right)C_a(d),
\qquad
C_a(d)=\sum_{t\in\Fq}\chi(t^2-d^2)\psi\!\left(\frac{dt^2}{2a}\right).
\]
The identity
\[
d(d^2+4s^2)
=
\frac23(\nu^3-\mu^3)-\frac1{24}(\nu-\mu)^3
\]
shows that $B_aB_a^*$ is unitarily conjugate, by the cubic diagonal multiplier $\psi(-\nu^3/(3a))$, to convolution with the kernel
\[
c_a(2d)
=
\frac1q\psi\!\left(\frac{d^3}{6a}\right)C_a(d).
\]
Its Fourier multiplier at $\lambda\in\Fq$ is
\begin{equation}
\widehat c_a(\lambda)
=
\frac1q
\left|
\sum_{u\in\Fq}
\chi(u)\psi\!\left(-\frac{u^3}{12a}+\lambda u\right)
\right|^2.
\label{eq:linear-symbol-multiplier}
\end{equation}
Indeed, in the double sum defining the multiplier use $u=t-d$, $v=t+d$; then $t^2-d^2=uv$ and
\[
\frac{dt^2}{2a}+\frac{d^3}{6a}
=
\frac{v^3-u^3}{12a},
\]
so the sum factors as a quantity times its complex conjugate. The inner sum in \eqref{eq:linear-symbol-multiplier} is the complete trace sum of the tensor product of the quadratic Kummer sheaf $\mathcal L_\chi$ with the Artin--Schreier sheaf attached to $-u^3/(12a)+\lambda u$. Since $p>3$ and the cubic coefficient is nonzero, this rank-one sheaf has bounded conductor and is not geometrically trivial. Deligne's theorem \cite{Deligne1980} therefore gives an $O(q^{1/2})$ bound, uniformly in $a$ and $\lambda$. Hence $\widehat c_a(\lambda)\ll1$ uniformly in $\lambda$, and consequently $\|B_aB_a^*\|\ll1$. This proves the lemma.
\end{proof}

\subsection{The degenerate quadratic case}

\begin{lemma}[Affine dependence and Vinh reduction]\label{lem:degenerate-quadratic}
Suppose that $Q=F'G'$ has degree $2$ and $\disc Q=0$. Then the conclusion of Theorem~\ref{thm:low-degree-incidence} holds.
\end{lemma}

\begin{proof}
Both $F'$ and $G'$ are nonconstant linear polynomials. Since their product has zero discriminant, they are proportional: $G'=\kappa F'$ for some $\kappa\in\Fq^\times$. Hence
\[
G=\kappa F+\delta
\]
for some $\delta\in\Fq$. The incidence equation becomes
\[
t-J(x)=(a+\kappa b)F(x)+\delta b+\lambda ab.
\]
Define
\[
(s,r)=(F(x),t-J(x)),
\qquad
(u,v)=(a+\kappa b,\delta b+\lambda ab).
\]
Then incidences become ordinary point--line incidences $r=us+v$.

The point map has fibres of size at most $2$, because $F$ is quadratic. For fixed $(u,v)$, substituting $a=u-\kappa b$ into the definition of $v$ gives
\[
\lambda\kappa b^2-(\delta+\lambda u)b+v=0,
\]
so the parameter map also has fibres of size at most $2$. Push the indicator functions of $\cP$ and $\Omega$ forward under these two maps. The resulting nonnegative integer weights are bounded by $2$, and their $\ell^1$ norms are $|\cP|$ and $|\Omega|$, respectively. Lemma~\ref{lem:weighted-vinh} gives
\[
\left|I(\cP,\Omega)-\frac{|\cP||\Omega|}{q}\right|
\le
q^{1/2}\|w_P\|_2\|w_L\|_2
\le
2q^{1/2}\sqrt{|\cP||\Omega|}.
\]
\end{proof}

\subsection{The nondegenerate quadratic symbol}

\begin{lemma}[Quadratic Weyl-symbol estimate]\label{lem:quadratic-symbol}
Assume $p>3$ and let $A,C\in\Fq^\times$. Then
\[
\left\|
\Op^W\bigl(\chi(Am^2+C-\xi^2)\bigr)
\right\|_{2\to2}
\ll1,
\]
with an absolute implied constant.
\end{lemma}

\begin{proof}
Expand the quadratic character by the Gauss identity:
\begin{equation}
\Op^W\bigl(\chi(Am^2+C-\xi^2)\bigr)
=
\mathfrak g^{-1}
\sum_{t\ne0}\chi(t)\psi(Ct)U_t,
\qquad
U_t=\Op^W\bigl(\psi(t(Am^2-\xi^2))\bigr).
\label{eq:quadratic-symbol-gauss}
\end{equation}
By Lemma~\ref{lem:weil-torus-input}, the parameters in the exceptional set $E_A$ contribute $O(1)$ to the norm of \eqref{eq:quadratic-symbol-gauss}, because $|E_A|\le2$, $\|U_t\|=q^{1/2}$ there, and $|\mathfrak g|=q^{1/2}$.

For $t\in\Fq^\times\setminus E_A$, write $U_t=\gamma_A(t)\rho(g_t)$ with $g_t\in T_A$. Since $T_A(\Fq)$ is finite abelian and $\rho$ is unitary, the representation space decomposes orthogonally into $T_A$-character spaces. On the $\nu$-isotypic space, the nonexceptional part of \eqref{eq:quadratic-symbol-gauss} acts by the scalar
\begin{equation}
\lambda_\nu
=
\mathfrak g^{-1}
\sum_{t\in\Fq^\times\setminus E_A}
\chi(t)\gamma_A(t)\nu(g_t)\psi(Ct).
\label{eq:torus-eigenvalue}
\end{equation}
By Lemma~\ref{lem:weil-torus-input}, the factor $\chi(t)\gamma_A(t)\nu(g_t)$ is the trace function of a tame rank-one sheaf with conductor bounded by an absolute constant. Since $C\ne0$, the Artin--Schreier factor $\psi(Ct)$ is wildly ramified at $\infty$ and cannot be cancelled by this tame factor. Deligne's theorem \cite{Deligne1980} therefore gives
\[
\left|
\sum_{t\in\Fq^\times\setminus E_A}
\chi(t)\gamma_A(t)\nu(g_t)\psi(Ct)
\right|
\ll q^{1/2},
\]
uniformly in $\nu$. Dividing by $|\mathfrak g|=q^{1/2}$ in \eqref{eq:torus-eigenvalue} yields $|\lambda_\nu|\ll1$. Taking the maximum over character spaces and adding the exceptional contribution proves the lemma.
\end{proof}

\subsection{Completion of the low-degree proof}

\begin{proof}[Proof of Theorem~\ref{thm:low-degree-incidence}]
If $Q$ in \eqref{eq:midpoint-factorization} is quadratic with zero discriminant, Lemma~\ref{lem:degenerate-quadratic} proves the theorem directly. We may therefore assume that $Q$ is constant, nonconstant linear, or quadratic with nonzero discriminant.

For $\tau\ne0$, put $c=\tau/\lambda$. Lemma~\ref{lem:weyl-symbol} expresses the Kloosterman block as $q\Op^W(b_c)$ plus at most two rank-one terms of norm $q$. If $Q$ is constant, $\|\Op^W(b_c)\|\le1$. If $Q$ is linear, a translation in $m$ and multiplication by a harmless quadratic-character constant reduce to Lemma~\ref{lem:linear-symbol}. If $Q$ is quadratic with nonzero discriminant, completing the square gives
\[
4cQ(m)=A(m-m_0)^2+C
\]
with $A,C\in\Fq^\times$. Translation in the physical variable is a unitary conjugation of Weyl quantization, so Lemma~\ref{lem:quadratic-symbol} applies. Consequently,
\[
\|M_\tau\|\ll q
\qquad(\tau\ne0).
\]
The zero block satisfies $\|M_0\|\ll q$ by Lemma~\ref{lem:kloosterman-blocks}. Hence the full extended cross-frequency matrix in Section~\ref{sec:general-proof} has norm $O(q)$. Lemma~\ref{lem:gram-decomposition} now gives
\[
\|\mathsf S\mathsf S^*\|\ll q^2,
\qquad
\|\mathsf S\|\ll q.
\]
Finally, Lemma~\ref{lem:output-fourier} yields $\|\mathsf B\|\ll q^{1/2}$, and pairing with indicator vectors proves the theorem.
\end{proof}

\section{Unrestricted expansion consequences at the Vinh scale}

\begin{proof}[Proof of Corollary~\ref{cor:low-degree-expansion}]
With $W=\Fq\setminus P(X,Y,Z)$, the case $W=\varnothing$ is trivial. Assume $W\ne\varnothing$, and take $\cP=X\times W$ and $\Omega=Y\times Z$ exactly as in the proof of Corollary~\ref{cor:general-expansion}. Then $I(\cP,\Omega)=0$, and Theorem~\ref{thm:low-degree-incidence} gives
\[
\frac{|X||W||Y||Z|}{q}
\ll
q^{1/2}\sqrt{|X||W||Y||Z|}.
\]
Thus
\[
|W|\ll\frac{q^3}{|X||Y||Z|},
\]
which is the first assertion. The product-size consequence follows on inserting $|X||Y||Z|\ge q^{2+\varepsilon}$.
\end{proof}

\begin{proof}[Proof of Corollary~\ref{cor:intrinsic-quadratic-expansion}]
Let $W=\Fq\setminus P(X,Y,Z)$. If $W=\varnothing$, there is nothing to prove. Otherwise take $\cP=X\times W$ and $\Omega=Y\times Z$. As before, $I(\cP,\Omega)=0$. Theorem~\ref{thm:intrinsic-quadratic} gives
\[
\frac{|X||W||Y||Z|}{q}
\ll
\sqrt{\kappa(F,G)}\,q^{1/2}\sqrt{|X||W||Y||Z|}.
\]
Squaring and cancelling the nonzero factor $|X||W||Y||Z|$ yields
\[
|W|
\ll
\kappa(F,G)\frac{q^3}{|X||Y||Z|}.
\]
The second assertion follows by inserting the stated product-size hypothesis.
\end{proof}

\section{Further remarks}

Theorem~\ref{thm:general-incidence} is insensitive to the degree of the common term $J$: it disappears from the cross-frequency spectral problem after diagonal unitary conjugation. The exponent $5/8$ arises from a fourth-moment estimate. Improving it for primitive coefficient curves of degree at least three would require additional control of spectral concentration beyond the two-pullback correlation used here. In the primitive quadratic range, the midpoint--difference factorization reduces the problem to a quadratic phase-space symbol, which is why the Weil representation recovers the Vinh scale.

Proposition~\ref{prop:primitive-factorization} and Theorem~\ref{thm:intrinsic-quadratic} show that raw polynomial degree mixes two different kinds of complexity: the intrinsic degree of the geometric coefficient curve and the multiplicity of its parametrization. The latter is spectrally inexpensive in the precise sense of Proposition~\ref{prop:finite-fibre-transfer}. Thus a natural continuation of the present work is to study the spectral problem for primitive coefficient curves of intrinsic degree three and higher, rather than to organize the problem only by the raw degrees of $F$ and $G$.

The tangent-plane formulation in Proposition~\ref{prop:tangent-planes} suggests a complementary geometric question: to what extent can strong mixing be obtained when points are constrained to a low-degree cylinder and the dual hyperplanes are constrained to a low-degree coefficient variety? The present family is a particularly rigid instance in which both restrictions can be exploited simultaneously.

\paragraph{\textbf{AI Assistance Statement}}

The author used ChatGPT, with GPT-5.6 available through a ChatGPT Plus subscription, as an auxiliary research and writing tool during parts of the exploratory development of this work. The tool was used for brainstorming possible proof strategies, checking intermediate calculations and logical dependencies, organizing arguments, and assisting with the exposition and preparation of the manuscript.

All mathematical statements, proofs, references, and final formulations appearing in this paper were reviewed and selected by the author. The author assumes full responsibility for the correctness, originality, and presentation of the results.

\end{document}